\documentclass[reqno,11pt]{amsart}
\usepackage{amsmath,amsfonts,amssymb,amsxtra,latexsym,amscd,enumerate,amsthm,verbatim}

\usepackage{hyperref}
\usepackage{graphicx}

\usepackage[margin=1.40in]{geometry}
\numberwithin{equation}{section}

\newcommand{\R}{\mathbb{R}}

\newcommand{\T}{\mathbb{T}}

\numberwithin{equation}{section} 
\newcommand{\lv}{\underline{v}}
\newcommand{\uv}{\overline{v}}

\newtheorem{theorem}{Theorem}[section]
\newtheorem{lemma}[theorem]{Lemma}
\newtheorem{proposition}[theorem]{Proposition}

\newtheorem{remark}[theorem]{Remark}

\newtheorem{claim}[theorem]{Claim}

\begin{document}

\title{linear inviscid damping in Gevrey spaces}


\author{Hao Jia}
\address{University of Minnesota}
\email{jia@umn.edu}


\begin{abstract}
{\small}
We prove linear inviscid damping near a general class of monotone shear flows in a finite channel, in Gevrey spaces. It is an essential step towards proving nonlinear inviscid damping for general shear flows that are not close to the Couette flow, which is a major open problem in 2d Euler equations. 
\end{abstract}

\maketitle

\setcounter{tocdepth}{1}


\section{Introduction}
\subsection{Main equations} 
Consider the two dimensional Euler equation linearized around a shear flow $(b(y),0)$, in the periodic channel $ (x,y)\in \mathbb{T}\times[0,1]$:
\begin{equation}\label{Main1}
\begin{split}
&\partial_t\omega+b(y)\partial_x\omega-b''(y)u^y=0,\\
&{\rm div}\,u=0\qquad{\rm and}\qquad \omega=-\partial_yu^x+\partial_xu^y,
\end{split}
\end{equation}
with the natural non-penetration boundary condition $u^y|_{y=0,1}=0$. 

For the linearized flow, 
$\int\limits_{\mathbb{T}\times[0,\,1]}u^x(x,y,t)\,dxdy$ and $ \int\limits_{\mathbb{T}\times[0,\,1]}\omega(x,y,t)\,dxdy$
are conserved quantities. In this paper, we will assume that
 $$\int_{\mathbb{T}\times[0,1]}u_0^x(x,y)\,dxdy=\int_{\mathbb{T}\times[0,1]}\omega_0\,dx dy=0.$$
 These assumptions can be dropped by adjusting $b(y)$ with a linear shear flow $C_0y+C_1$.
 Then one can see from the divergence free condition on $u$ that 
there exists a stream function $\psi(t,x,y)$ with $\psi(t,x,0)=\psi(t,x,1)\equiv 0$, such that 
\begin{equation}\label{eqS1}
u^x=-\partial_y\psi,\,\,u^y=\partial_x\psi.
\end{equation}
The stream function $\psi$ can be solved through
\begin{equation}\label{eq:equationStream}
\Delta\psi=\omega, \qquad \psi|_{y=0,1}=0.
\end{equation}
We summarize our equations as follows
\begin{equation}\label{main}
\left\{\begin{array}{ll}
\partial_t\omega+b(y)\partial_x\omega-b''(y)\partial_x\psi=0,&\\
\Delta \psi(t,x,y)=\omega(t,x,y),\qquad \psi(t,x,0)=\psi(t,x,1)=0,&\\
(u^x,u^y)=(-\partial_y\psi,\partial_x\psi),&
\end{array}\right.
\end{equation}
for $ t\ge0, (x,y)\in\mathbb{T}\times[0,1]$. 

Our goal is to understand the long time behavior of $\omega(t)$ in Gevrey spaces as $t\to\infty$, with Gevrey initial $\omega_0$. 

\subsection{The linear inviscid damping}

Hydrodynamical stability is a classical topic in mathematical analysis of fluid flows, pioneered by prominent figures such as Rayleigh \cite{Rayleigh}, Kelvin \cite{Kelvin}, Orr \cite{Orr}, among many others. The main focus was to study stability of important physically relevant flows, such as shear flows and vortices.

 In this paper we consider shear flows. There are extensive works on the linear stability property of these flows. In particular, Rayleigh \cite{Rayleigh} proved that shear flows with no inflection points are spectrally stable. Orr \cite{Orr} in 1907 observed the $t^{-1}$ decay rate of the velocity when the shear flow is Couette (linear shear), and Case \cite{Case} provided a formal proof in the case of a finite channel. See also Lin and Zeng \cite{ZhiWu} for a sharp version with optimal dependence on the regularity of the initial data. 
 
 The observation of Orr can be described roughly as follows. Consider the linearized equation near Couette flow:
$$\partial_t\omega+y\partial_x\omega=0, \,\,(x,y)\in\mathbb{T}\times \R.$$

One can solve this equation explicitly and it follows that
$\omega(t,x,y)=\omega_0(x-yt,y).$
The equation for the stream function becomes
$\Delta\psi(t,x,y)=\omega(t,x,y)=\omega_0(x-yt,y)$
for $(x,y)\in \mathbb{T}\times \R$ and therefore
\begin{equation}\label{decaycostderivative}
\widetilde{\psi}(t,k,\xi)=-\frac{\widetilde{\omega}(t,k,\xi)}{k^2+|\xi|^2}=-\frac{\widetilde{\omega_0}(k,\xi+kt)}{k^2+|\xi|^2}.
\end{equation}
In the above, $\widetilde{h}$ denotes the Fourier transform of $h$ in $x,\,y$.
Assume that $\omega_0$ is smooth, so $\widetilde{\omega}_0(k,\xi)$ decays fast in $k,\,\xi$. Then we can view $\xi$ as
$$\xi=-kt+O(1),$$ 
and hence $\widetilde{\psi}(t,k,\xi)$ decays like $|k|^{-2}\langle t\rangle^{-2}$ for each $k\neq 0$. Similarly, using the relations $u^x=-\partial_y\psi$ and $u^y=\partial_x\psi,$
we conclude that $\widetilde{u^x}$ decays like $|k|^{-1}\langle t\rangle^{-1}$ and $\widetilde{u^y}$ decays like $|k|^{-1}\langle t\rangle^{-2}$ for all $k\neq0$. Hence, the velocity field decays to another shear flow $(u_{\infty}(y),0)$. 


 For general monotone shear flows, the linearized operator becomes more complicated due to the extra term $b''(y)\partial_x\psi$, see \eqref{main}, which can not be treated as perturbations. Therefore spectral analysis of the linearized operator is required to understand the dynamical properties of the associated flow. For results on the general spectral property of the linearized operator, we refer to Faddeev \cite{Faddeev} and Lin \cite{Lin}. In the direction of inviscid damping, Stepin \cite{Stepin} proved $t^{-\nu}$ decay of the stream function associated with the continuous spectrum, Rosencrans and Sattinger \cite{Rosencrans} proved $t^{-1}$ decay for analytic monotone shear flows. 
 
 Recently, inspired by the remarkable work of Bedrossian and Masmoudi \cite{BeMa} on the nonlinear asymptotic stability of shear flows close to the Couette flow in $\mathbb{T}\times\R$ (see also an extension \cite{IOJI} to $\mathbb{T}\times[0,1]$), optimal decay estimates for the linear problem received much attention, see e.g.  Zillinger \cite{Zillinger1,Zillinger2} and references therein for shear flows close to Couette.  In an important work, Wei, Zhang and Zhao \cite{dongyi} obtained the optimal decay estimates for the linearized problem around monotone shear flows, under very general conditions. In \cite{JiaL} the author identified the main term in the asymptotics of the stream function. From the perspective of the linearized problem, the works \cite{JiaL,dongyi} provided a quite satisfactory picture for the linear inviscid damping problem, in Sobolev spaces.

We also refer the reader to important developments for the linear inviscid damping in the case of non-monotone shear flows \cite{Bouchet,Dongyi2, Dongyi3} and circular flows \cite{Bed2,Zillinger3}. See also Grenier et al \cite{Grenier} for an approach using methods from the study of Schr\"odinger operators. 

\subsection{Nonlinear inviscid damping and Gevrey regularity}
The nonlinear asymptotic stability of shear flows is much more subtle and challenging. So far, the only nonlinear asymptotic stability results are \cite{BeMa} by  Bedrossian and Masmoudi for plane Couette flows, and the extension by Ionescu and the author \cite{IOJI} to a finite channel (thus considering finite energy solutions and boundary effects) still for Couette flows. 

One of the main difficulties in proving nonlinear stability is the presence of ``resonances" in the nonlinearity which can accumulate over time.  To control the resonances, very high (in fact Gevrey) regularity of the initial data is required. A key original idea, introduced in Bedrossian and Masmoudi \cite{BeMa}, was to use time-dependent energy functionals associated with  \emph{imbalanced} weights to  control the vorticity in suitable nonlinearly adapted coordinates. The energy functionals are carefully designed and lose derivatives in specific ways to balance the resonances. The total loss of  regularity over time $t\in[0,\infty)$ in this procedure is Gevrey-2, and thus one needs to work with at least Gevrey-2 regular functions to maintain meaningful control over the final profile at time $\infty$.  We refer to \cite{BeMa} and \cite{IOJI} for detailed discussions on the nonlinear problem. 

It is clear on the technical level that Gevrey space regularity is necessary for the proofs in \cite{BeMa} and \cite{IOJI}. However, the requirement of Gevrey regular initial data is not just technical. In a recent work, Deng and Masmoudi \cite{Deng} demonstrated that the inviscid damping (with the precise control as in \cite{BeMa} and \cite{IOJI}) does not hold with initial data which is only logarithmically rougher than Gevrey-2. In low Sobolev spaces we have more definitive counterexamples to inviscid damping, see \cite{ZhiWu}.  The celebrated work of Mouhot and Villani \cite{Villani} on Landau damping, where decay also comes from mixing,  requires similar Gevrey regularity on the initial data.  

Therefore, to prove nonlinear inviscid damping, the linear stability analysis needs to be performed in Gevrey spaces, and the methods for proving such Gevrey estimates need to be flexible enough so that one can work with the specific weights used in the nonlinear analysis. 

For the Couette case, the linearized problem can be explicitly solved and it is not an issue to work in Gevrey spaces. In the case of more general monotone shear flows though, this is not the case and it has been an important open problem to study linear inviscid damping in Gevrey spaces. Linear inviscid damping in high regularity spaces has been studied in other contexts, see e.g. Bedrossian-Coti Zelati-Vicol \cite{Bed2} for the 2D vortices, where significant efforts were devoted to study scattering in high Sobolev spaces and the need to work in Gevrey spaces was commented on.

\vspace{-0.1cm}
\subsection{The main results}
In this paper, we prove linear inviscid damping and scattering of the vorticity in Gevrey spaces for a general class of monotone shear flows which need not be close to the Couette flow. As far as we know, this is the first result of linear inviscid damping in Gevrey spaces for general shear flows. In addition, our method is flexible enough and we can work with the specific weights in \cite{IOJI}.  Those weights are refined versions of the weights introduced in \cite{BeMa} and have the necessary smoothness to implement the ideas here. 
We believe the techniques introduced in this paper will play a crucial role in establishing nonlinear inviscid damping near general monotone shear flows. 

We now describe more precisely the main assumptions and our main conclusion. The main conditions we shall assume on the shear flow $b(y)\in C^{\infty}([0,1])$ are:\\

\hspace{0.1in}(1)  For some $\vartheta_0\in(0,1/10)$, 
\begin{equation}\label{B}
\vartheta_0/100\leq b'(y)\leq 1/(100 \vartheta_0)\qquad {\rm and}\qquad b''(y)\equiv 0\,\,{\rm for}\,\,y\in[0,\vartheta_0]\cup[1-\vartheta_0,1],
\end{equation}
\qquad \,\,\,and for some $s\in(0,1)$,
\begin{equation}\label{B1}
\|b\|^2_{L^{\infty}(0,1)}+\int_{\R}e^{2\vartheta_0\langle \xi\rangle^{(s+1)/2}}\left|\widetilde{\,b''\,}(\xi)\right|^2d\xi<1/\vartheta_0.
\end{equation}

\begin{equation}\label{Sp} {\rm (2)\,\, The\,\, linearized\,\, operator}\,\, \omega\to b(y)\partial_x\omega-b''(y)\psi\,\,{\rm has\,\, no\,\, embedded\,\, eigenvalues}.\end{equation}\\
In \eqref{B1} and the rest of the paper we use $\widetilde{\,\,\,\,\,\,\,}$ to denote Fourier transform in $\R$ or $\mathbb{T}\times\R$. We assume that $b'>0$ for the sake of concreteness. The other case $b'<0$ can be treated completely analogously, with a change of time $t\to-t$. Throughout the paper we fix $s\in(0,1)$ from \eqref{B1}.

Our main result is the following theorem.
\begin{theorem}\label{thm}
 Suppose that $\omega_0\in C^{\infty}(\mathbb{T}\times [0,1])$ satisfies ${\rm supp}\,\omega_0\subseteq \mathbb{T}\times[\vartheta_1,1-\vartheta_1]$ for some $\vartheta_1\in(0,\vartheta_0)$. Let $\omega$ be the smooth solution to \eqref{main} with initial data $\omega_0$ and let $\psi$ be the associated stream function. Fix a smooth cutoff function
 $\varphi(y)$ satisfying $\varphi\equiv 1$ on $[\vartheta_1/2,1-\vartheta_1/2]$ and ${\rm supp}\,\varphi\subseteq [\vartheta_1/3,1-\vartheta_1/3]$ and $$\int_{\R}e^{2\langle \xi\rangle^{(1+s)/2}}|\widetilde{\varphi}(\xi)|^2\,d\xi\lesssim_{\vartheta_1,s}1.$$

 Define the change of variables
 \begin{equation}\label{U1}
 z=x-tv,\qquad v=b(y),\qquad{\rm for}\,\,x\in\mathbb{T}, y\in[0,1],
 \end{equation}
 and
 \begin{equation}\label{U2}
 f(t,z,v):=\omega(t,x,y),\qquad \phi(t,z,v):=\psi(t,x,y),\qquad f_0(z,v):=\omega_0(x,y),\qquad \Psi(v):=\varphi(y).
 \end{equation}
  
 Assume that for some $\lambda\in(0,\infty)$,
\begin{equation}\label{B2}
\Lambda^2:=\sum_{k\in\mathbb{Z}\backslash\{0\}}\int_{\R}e^{2\lambda\langle k,\xi\rangle^s}\left|\widetilde{f_0}(k,\xi)\right|^2d\xi<\infty.
\end{equation}
 Then 

(i) the localized stream function $\Psi \phi$ satisfies 
\begin{equation}\label{U3}
\left|\widetilde{\,\,\Psi\phi\,\,}(t,k,\xi)\right|\lesssim_{\lambda,\vartheta_1,s}\frac{e^{-\lambda \langle k,\xi\rangle^s}}{k^2+(\xi-kt)^2}\Lambda,\qquad {\rm for}\,\,k\in\mathbb{Z}\backslash\{0\}, \xi\in\R;
\end{equation}

(ii) $f(t)$ is compactly supported in $\mathbb{T}\times[b(\vartheta_1),b(1-\vartheta_1)]$ and  satisfies for all $t\ge0$,
\begin{equation}\label{U4}
\sum_{k\in\mathbb{Z}\backslash\{0\}}\int_{\R} e^{2\lambda\langle k,\xi\rangle^{s}}\left|\widetilde{f}(t,k,\xi)\right|^2d\xi\lesssim_{\lambda,\vartheta_1,s} \sum_{k\in\mathbb{Z}\backslash\{0\}}\int_{\R}e^{2\lambda\langle k,\xi\rangle^s}\left|\widetilde{f_0}(k,\xi)\right|^2d\xi;\end{equation}

(iii) $\lim_{t\to\infty}f(t)=f_{\infty}$ exists and satisfies for $\lambda'\in(0,\lambda)$,
\begin{equation}\label{U5}
\left\|e^{\lambda'\langle k,\xi\rangle^s}\Big[\widetilde{f}(t,k,\xi)-\widetilde{f_{\infty}}(k,\xi)\Big]\right\|_{L^2(k\in\mathbb{Z}\backslash\{0\},\xi\in\R)}\lesssim_{\lambda,\vartheta_1,s} \frac{1}{\langle t\rangle} \left\|e^{\lambda\langle k,\xi\rangle^s}\widetilde{f_0}(k,\xi)\right\|_{L^2(k\in\mathbb{Z}\backslash\{0\},\xi\in\R)}.
\end{equation}
\end{theorem}

\smallskip
\noindent
\begin{remark}
The assumptions on the compact support of $b''$ and $\omega_0$ (or vanishing with infinite order at the boundary) are likely necessary to prove scattering in Gevrey spaces. Indeed, Zillinger  \cite{Zillinger2} showed that scattering does not hold in high Sobolev spaces if one does not assume the vorticity to vanish at the boundary. The boundary effect can also be seen clearly in \cite{JiaL} in the main asymptotic term for the stream function. The assumptions on the support of $b''$ is required to preserve the compact support of $\omega(t)$ in $\mathbb{T}\times(0,1)$.

\end{remark}

\begin{remark}\label{RSp}
There are a large class of shear flows $b$ satisfying our assumptions. For instance,  for any $b(y)$ which satisfies $|b'|\ge 1$ and $|b'''|<1$, the spectral assumption \eqref{Sp} is satisfied. 
\end{remark}

\begin{remark}
The evolution of \,\,$0$ mode of $\omega$ in $x$ is trivial and hence we removed it from the statement of Theorem \ref{thm}. We note that the change of variable \eqref{U1}-\eqref{U2} is reminiscent of the nonlinear change of coordinates introduced in \cite{BeMa}. Such change of variables are essential to re-normalize the loss of regularity of $\omega$ due to transport in $x$, see \eqref{main}. It is therefore natural, in view of applications to nonlinear analysis, to work in the new variables $z,v$. 

\end{remark}

\subsection{Main ideas of the proof} 
We now outline some of the key ideas used in the proof of Theorem \ref{thm}. The main task is to prove the bounds \eqref{U3} on the localized stream function in the new coordinates. The bounds \eqref{U4}-\eqref{U5} follow relatively easily from \eqref{U3} (with suitable adjustments), in view of the equation \eqref{main} for $\omega$. 

To prove \eqref{U3}, we first use standard spectral representation formula to express $\psi$ as oscillatory integral of the generalized eigenfunctions in the spectral parameter, see \eqref{F3.4}-\eqref{F7}. The main difficulty in proving \eqref{U3} is that the generalized eigenfunctions are not smooth, which can be seen from the equation \eqref{F7} due to the presence of the singular factor $\frac{1}{b(y)-b(y_0)+i\epsilon}$ when $y=y_0$. We note that the generalized eigenfunctions are parametrized by the space variable $y$, the spectral variable $y_0$, together with a smoothing variable $\epsilon\to0$.

The essential new idea of our paper is to correctly capture the singularity of the generalized eigenfunctions in the variables $v=b(y), w=b(y_0)$. The introduction of the new coordinates, which is reminiscent of the nonlinear change of variables introduced in \cite{BeMa}, simplifies the main singular factor from  $\frac{1}{b(y)-b(y_0)+i\epsilon}$ to $\frac{1}{v-w+i\epsilon}$. 

Even in the new coordinates the generalized eigenfunctions are still singular in both $v$ and $w$. To extract the precise singular behavior, we shift the coordinate, using the transform $(v,w)\to(v+w,w)$. This shift of coordinate, though simple, captures the essence of the singularity. The generalized eigenfunctions after the shift of coordinates then become \emph{smooth} in $w$ and all the singularity is now transferred to the variable $v$, which is heuristically clear since the main singular factor is now transformed to $\frac{1}{v+i\epsilon}$. 

To establish the smoothness of the generalized eigenfunctions in $w$, we use the limiting absorption principle, and a two-step approach by considering separately the high frequency and low frequency cases. The assumption of no embedded eigenvalues implies suitable estimates on the generalized eigenfunctions in low Sobolev spaces, which is sufficient for control on the low frequencies in $w$. To pass to control on the high frequencies in $w$, we apply suitable Fourier multiplier operator and estimate the resulting functions, using the bounds coming from the limiting absorption principle. Since the coefficients are not constant functions, we need to estimate a number of commutators by showing that they are perturbative in comparison with the main terms. Similar ideas have recently played a crucial role in \cite{Jia3}. In our case, the implementation is much simpler since our Fourier multipliers are smooth and explicitly given. The complication here is that we need to combine the Fourier analysis with the spectral analysis, which seems to be of independent interest and may be useful for other problems.

\subsection{Organization and Notations} 
The rest of the paper is organized as follows.
\begin{itemize}
 \item In section \ref{rep} we use general spectral projection to derive the representation formula for the stream function in terms of generalized eigenfunctions; \\
\item In section \ref{lim} we establish the limiting absorption principle for the generalized eigenfunctions in $H^1$; \\
\item In the main section \ref{gev} we prove Gevrey bounds for the re-normalized generalized eigenfunctions in the variable $w$; \\
\item In section \ref{evo} we prove the main theorem using the Gevrey bounds from section \ref{gev}; \\
\item In the appendix we recall some basic properties of Gevrey spaces and prove an important bound for the localized Green's function in Gevrey spaces.
\end{itemize}

We often use the notation $A\lesssim B$ to denote $A\leq C B$ with a constant $C$ which can only depend on fixed parameters such as $\lambda,\vartheta_1,s$.

\section{Representation formula for the stream function $\psi$}\label{rep}
Taking Fourier transform in $x$ in the equation \eqref{main} for $\omega$, we obtain that
 \begin{equation}\label{F3.0}
 \partial_t\omega_k+ikb(y)\omega_k-ikb''(y)\psi_k=0,
 \end{equation}
 for $k\in\mathbb{Z}, t\ge0, y\in[0,1]$. In the above, $\omega_k$ and $\psi_k$ are the $k$-th Fourier coefficients for $\omega,\psi$ respectively.
 For each $k\in\mathbb{Z}\backslash\{0\}$, we set for any $g\in L^2(0,1)$,
 \begin{equation}\label{F3.1}
 L_kg(y)=b(y)g(y)+b''(y)\int_0^1G_k(y,z)g(z)dz,
 \end{equation}
 where $G_k$ is the Green's function for the operator $k^2-\frac{d^2}{dy^2}$ on $(0,1)$ with zero Dirichlet boundary condition. 
Then \eqref{F3.0} can be reformulated as
 \begin{equation}\label{F3.2}
 \partial_t\omega_k+ikL_k\omega_k=0.
 \end{equation}
 
 The general spectral property of $L_k$ is well understood, and the spectrum is in general consisted of the continuous spectrum $[b(0), b(1)]$ with possible embedded eigenvalues at the inflection points of $b(y)$, i.e. points $y_c$ where $b''(y_c)=0$, together with some discrete eigenvalues with nonzero imaginary part which can only accumulate at embedded eigenvalues, for small $k$. See for instance \cite{Faddeev}. 

In view of \eqref{Sp}, the operator $g\in L^2\to b(y)g+b''(y)\int_0^1G_k(y,z)g(z)\,dz\in L^2$ has no eigenvalues with
the value $b(y_0), y_0\in[0,1]$,  for any $ k\in\mathbb{Z}\backslash\{0\}.$ As discussed in Remark \ref{RSp}, the spectral condition \eqref{Sp} is satisfied by a large class of shear flows $b$.

 \begin{proposition}\label{F3.3}
 Suppose that the spectral condition \eqref{Sp} holds and that $\omega_0$ has trivial projection in the discrete modes. Then the stream function $\psi_k(t,y)$, $k\in\mathbb{Z}\backslash\{0\}, y\in[0,1]$ has the representation 
 \begin{equation}\label{F3.4}
 \psi_k(t,y)=-\frac{1}{2\pi i}\lim_{\epsilon\to0+}\int_{0}^1e^{-ikb(y_0) t}b'(y_0)\left[\psi_{k,\epsilon}^{-}(y,y_0)-\psi_{k,\epsilon}^{+}(y,y_0)\right]dy_0,
 \end{equation}
 where $\psi_{k,\epsilon}^{\iota}(y,y_0)$ for $\iota\in\{+,-\}, y,y_0\in[0,1],k\in\mathbb{Z}\backslash\{0\}$,  and  sufficiently small $\epsilon\in[-1/4,1/4]\backslash\{0\}$, are the solutions to
 \begin{equation}\label{F7}
 \begin{split}
  -k^2\psi_{k,\epsilon}^{\iota}(y,y_0)+\frac{d^2}{dy^2}\psi_{k,\epsilon}^{\iota}(y,y_0)-\frac{b''(y)}{b(y)-b(y_0)+i\iota\epsilon}\psi_{k,\epsilon}^{\iota}(y,y_0)=\frac{-\omega^k_0(y)}{b(y)-b(y_0)+i\iota\epsilon}.
 \end{split}
 \end{equation}
 \end{proposition}
 
 \begin{proof}
By standard theory of spectral projection, we  have
 \begin{equation}\label{F4}
 \begin{split}
 \omega_k(t,y)&=\frac{1}{2\pi i}\lim_{\epsilon\to0+}\int_{\R}e^{i\lambda t}\left[(\lambda+kL_k-i\epsilon)^{-1}-(\lambda+kL_k+i\epsilon)^{-1}\right]\omega^k_0\,d\lambda\\
 &=\frac{1}{2\pi i}\lim_{\epsilon\to0+}\int_{0}^1e^{-ikb(y_0) t}b'(y_0)\left[(-b(y_0)+L_k-i\epsilon)^{-1}-(-b(y_0)+L_k+i\epsilon)^{-1}\right]\omega^k_0\,dy_0.
 \end{split}
 \end{equation}
We then obtain
 \begin{equation}\label{F5}
 \begin{split}
 \psi_k(t,y)&=-\frac{1}{2\pi i}\lim_{\epsilon\to0+}\int_{0}^1e^{-ikb(y_0) t}b'(y_0)\int_0^1G_k(y,z)\\
 &\hspace{1in}\times\bigg\{\Big[(-b(y_0)+L_k-i\epsilon)^{-1}-(-b(y_0)+L_k+i\epsilon)^{-1}\Big]\omega^k_0\bigg\}(z)\,dz dy_0\\
 &=-\frac{1}{2\pi i}\lim_{\epsilon\to0+}\int_{0}^1e^{-ikb(y_0) t}b'(y_0)\left[\psi_{k,\epsilon}^{-}(y,y_0)-\psi_{k,\epsilon}^{+}(y,y_0)\right]dy_0.
 \end{split}
 \end{equation}
 In the above, 
 \begin{equation}\label{F6}
 \begin{split}
 &\psi_{k,\epsilon}^{+}(y,y_0):=\int_0^1G_k(y,z)\Big[(-b(y_0)+L_k+i\epsilon)^{-1}\omega^k_0\Big](z)\,dz,\\
 &\psi_{k,\epsilon}^{-}(y,y_0):=\int_0^1G_k(y,z)\Big[(-b(y_0)+L_k-i\epsilon)^{-1}\omega^k_0\Big](z)\,dz.
 \end{split}
 \end{equation}
Clearly $\psi_{k,\epsilon}^{+}(y,y_0), \psi_{k,\epsilon}^{-}(y,y_0)$ satisfy \eqref{F7}. The proposition is now proved.

\end{proof}

 
\begin{remark}\label{F6.1}
The existence of $\psi^{\iota}_{k,\epsilon}$ for sufficiently small $\epsilon\neq0$ follows from our spectral assumptions, which imply the solvability of \eqref{F7} for sufficiently small $\epsilon\neq0$, see \eqref{T6}.

By the assumption that ${\rm supp}\,b''\subseteq[\vartheta_1,1-\vartheta_1]$, it is clear that to study the evolutionary equation \eqref{main} for $\omega$, it suffices to study $\varphi(y)\psi_k(t,y)$.

\end{remark}

\section{The limiting absorption principles}\label{lim}

\subsection{Elementary properties of the Green's function}\label{sub1}
For integers $k\in\mathbb{Z}\setminus\{0\}$, recall that the Green's function $G_k(y,z)$ solves  
\begin{equation}\label{eq:Helmoltz}
-\frac{d^2}{dy^2}G_k(y,z)+k^2G_k(y,z)=\delta_z(y),
\end{equation}
with Dirichlet boundary conditions $G_k(0,z)=G_k(1,z)=0$, $z\in [0,1]$. $G_k$ has the explicit formula 
\begin{equation}\label{eq:GreenFunction}
G_k(y,z)=\frac{1}{k\sinh k}
\begin{cases}
\sinh(k(1-z))\sinh (ky)\qquad&\text{ if }y\leq z,\\
\sinh (kz)\sinh(k(1-y))\qquad&\text{ if }y\geq z,
\end{cases}
\end{equation}
and the symmetry
\begin{equation}\label{Gk2}
G_k(y,z)=G_k(z,y), \qquad {\rm for}\,\,k\in\mathbb{Z}\backslash\{0\}, y, z\in[0,1].
\end{equation}

We note the following bounds for $G_k$
\begin{equation}\label{Gk1.1}
\begin{split}
&\sup_{y\in[0,1], |A|\leq10}\bigg[|k|^2\big\|G_k(y,z)(\log{|z-A|})^{m}\big\|_{L^1(z\in[0,1])}+|k|\big\|\partial_{y,z}G_k(y,z)(\log{|z-A|})^{m}\big\|_{L^1(z\in[0,1])}\bigg]\\
&\qquad+\sup_{y\in[0,1],\alpha\in\{0,1\}}\bigg[|k|^{3/2-\alpha}\left\|\partial_{y,z}^{\alpha}G_k(y,z)\right\|_{L^2(z\in[0,1])}\bigg]\lesssim |\log{\langle k\rangle}|^m,\qquad {\rm for}\,\,m\in\{0,1,2,3\}.
\end{split}
\end{equation}

Define
\begin{equation}\label{bX5}
G_k'(y,z)=\frac{1}{\sinh{k}}\left\{\begin{array}{lr}
                                         -k\cosh{(k(1-z))}\cosh{(ky)}, &0\leq y\leq z\leq 1;\\
                                         -k\cosh{(kz)}\cosh{(k(1-y))}, &1\ge y>z\ge0.
                                          \end{array}\right.
\end{equation}
We note that
\begin{equation}\label{Gk1.2}
\partial_y\partial_zG_k(y,z)=\partial_z\partial_yG_k(y,z)=\delta(y-z)+G'_k(y,z),\qquad{\rm for}\,\,y,z\in[0,1].
\end{equation}

By direct computation, we see $G_k'$ satisfies the bounds
\begin{equation}\label{Gk3.1}
\begin{split}
&\sup_{y\in[0,1], |A|\leq10}\bigg[\big\|G_k'(y,z)(\log{|z-A|})^{m}\big\|_{L^1(z\in[0,1])}+|k|^{-1}\big\|\partial_{y,z}G_k'(y,z)(\log{|z-A|})^{m}\big\|_{L^1(z\in[0,1])}\bigg]\\
&\quad+\sup_{y\in[0,1],\alpha\in\{0,1\}}\bigg[|k|^{-1/2-\alpha}\left\|\partial_{y,z}^{\alpha}G_k'(y,z)\right\|_{L^2(z\in[0,1])}\bigg]\lesssim |\log{\langle k\rangle}|^m,\qquad {\rm for}\,\,m\in\{0,1,2,3\}.
\end{split}
\end{equation}

\subsection{The limiting absorption principle in the variable $y$}

Fix $\epsilon\in[-1/4,1/4]\backslash\{0\}, y_0\in[0,1],  k\in\mathbb{Z}\backslash\{0\}$. Recall the definition of the cutoff function $\varphi$ from Theorem \ref{thm} and define for each $g\in L^2(0,1)$ the operator
\begin{equation}\label{T1}
T_{k,y_0,\epsilon}g(y):=\int_{\R}\varphi(y)G_k(y,z)\frac{b''(z)g(z)}{b(z)-b(y_0)+i\epsilon}dz.
\end{equation}

To obtain the optimal dependence on the frequency variable $k$, we define
\begin{equation}\label{T1.1}
\|g\|_{H^1_k(\R)}:=\|g\|_{L^2(\R)}+|k|^{-1}\|g'\|_{L^2(\R)}.
\end{equation}

\begin{lemma}\label{T2}
For  $\epsilon\in[-1/4,1/4]\backslash\{0\}, y_0\in[0,1],  k\in\mathbb{Z}\backslash\{0\}$, the operator $T_{k,y_0,\epsilon}$ satisfies the bound
\begin{equation}\label{T3}
\|T_{k,y_0,\epsilon}g\|_{H^1_k(\R)}\lesssim |k|^{-1/3}\|g\|_{H^1_k(\R)},\qquad {\rm for\,\,all}\,\,g\in H^{1}_k(\R).
\end{equation}
In addition, we have the more precise regularity structure
\begin{equation}\label{T3.1}
\left\|\partial_yT_{k,y_0,\epsilon}g(y)+\varphi(y)\frac{b''(y)g(y)}{b'(y)}\log{(b(y)-b(y_0)+i\epsilon)}\right\|_{W^{1,1}(\R)}\lesssim \langle k\rangle^{3/4}\|g\|_{H^1_k(\R)}.
\end{equation}

\end{lemma}

\begin{proof}
Using integration by parts, we obtain 
\begin{equation}\label{T4}
\begin{split}
T_{k,y_0,\epsilon}g(y)&=\int_{\R}\varphi(y)G_k(y,z)\frac{b''(z)g(z)}{b'(z)}\partial_z\log{(b(z)-b(y_0)+i\epsilon)}dz\\
&=-\int_{\R}\varphi(y)\partial_zG_k(y,z)\frac{b''(z)g(z)}{b'(z)}\log{(b(z)-b(y_0)+i\epsilon)}dz\\
&\quad-\int_{\R}\varphi(y)G_k(y,z)\partial_z\left[\frac{b''(z)g(z)}{b'(z)}\right]\log{(b(z)-b(y_0)+i\epsilon)}dz,
\end{split}
\end{equation}
and
\begin{equation}\label{T4.1}
\begin{split}
&\partial_yT_{k,y_0,\epsilon}g(y)=-\int_{\R}\varphi'(y)\partial_zG_k(y,z)\frac{b''(z)g(z)}{b'(z)}\log{(b(z)-b(y_0)+i\epsilon)}dz\\
&-\varphi(y)\frac{b''(y)g(y)}{b'(y)}\log{(b(y)-b(y_0)+i\epsilon)}-\int_{\R}\varphi(y)G_k'(y,z)\frac{b''(z)g(z)}{b'(z)}\log{(b(z)-b(y_0)+i\epsilon)}dz\\
 &-\int_{\R}\partial_y\left[\varphi(y)G_k(y,z)\right]\partial_z\left[\frac{b''(z)g(z)}{b'(z)}\right]\log{(b(z)-b(y_0)+i\epsilon)}dz.
\end{split}
\end{equation}
Then \eqref{T3}-\eqref{T3.1} follow directly from the properties of Green's functions summarized in subsection \ref{sub1}, Minkowski and H\"older inequalities.
\end{proof}

We now prove the limiting absorption principle, using the assumption that there is no embedded eigenvalues.

\begin{lemma}\label{T5}
For sufficiently small and non-zero $\epsilon$, the following bounds hold
\begin{equation}\label{T6}
\|g+T_{k,y_0,\epsilon}g\|_{H^1_k(\R)}\gtrsim \|g\|_{H^1_k(\R)},
\end{equation}
for all  $y_0\in[0,1],  k\in\mathbb{Z}\backslash\{0\}$ and any $g\in H^1_k(\R)$.
\end{lemma}

\begin{proof}
We prove \eqref{T6} by contradiction. Thus we assume that there exist for $j\ge1$, a sequence of numbers $k_j\in\mathbb{Z}\backslash\{0\}$, $y_j\in[0,1]$, $\epsilon_j\in\mathbb{\R}\backslash\{0\}\to 0$ and functions  $g_j\in H_{k_j}^1(\R)$ with $\|g_j\|_{H_{k_j}^1(\R)}=1$, satisfying 
$k_j\to k_{\ast}\in(\mathbb{Z}\backslash\{0\})\cup \{\pm\infty\}$, $y_j\to y_{\ast}\in[0,1]$ as $j\to\infty$, such that
\begin{equation}\label{T6.1}
\left\|g_j+T_{k_j,y_j,\epsilon_j}g_j\right\|_{H_{k_j}^1(\R)}\to 0,\qquad{\rm as}\,\,j\to\infty.
\end{equation}
The bounds \eqref{T3} and \eqref{T6.1} imply that $|k_j|\lesssim 1$. Thus $k_{\ast}\in \mathbb{Z}\backslash\{0\}$. In addition, using $\|g_j\|_{H_{k_j}^1(\R)}=1$, the bounds \eqref{T3.1}, and noting the compact support property of $T_{k_j,y_j,\epsilon_j}g_j$, by passing to a subsequence, we can assume that $g_j\to g$ in $H^1(\R)$, where $\|g\|_{H^1_{k_{\ast}}}=1$.

In view of the formula \eqref{T4}, we obtain from \eqref{T6.1} that
\begin{equation}\label{T6.2}
g(y)+\lim_{j\to\infty}\int_{\R}\varphi(y)G_{k_{\ast}}(y,z)\frac{b''(z)g(z)}{b(z)-b(y_{\ast})+i\epsilon_j}\,dz=0.
\end{equation}
Hence we can assume that $g=\varphi h$ with $h\in H_0^1(\R)$ and $h\not\equiv0$. (We recall that by the definitions $\phi\equiv1$ on the support of $b''$.) In addition,
\begin{equation}\label{F6.3}
h(y)+\lim_{j\to\infty}\int_{\R}G_{k_{\ast}}(y,z)\frac{b''(z)h(z)}{b(z)-b(y_{\ast})+i\epsilon_j}\,dz=0,
\end{equation}
first on the support of $\varphi$, in view of \eqref{T6.2}. Then by re-defining $h$ outside of the support of $\varphi$ which does not affect the value of the integral in \eqref{F6.3} thanks to the assumption on the support of $b''$, we can assume \eqref{F6.3} holds for $y\in[0,1]$.
Applying $k_{\ast}^2-\frac{d^2}{dy^2}$ to \eqref{F6.3}, we get
\begin{equation}\label{F6.4}
k_{\ast}^2h(y)-h''(y)+\lim_{j\to\infty}\frac{(b(y)-b(y_{\ast}))b''(y)h(y)}{(b(y)-b(y_{\ast}))^2+\epsilon_j^2}+iC_{y_{\ast}}b''(y_{\ast})h(y_{\ast})\delta(y-y_{\ast})=0,
\end{equation}
in the sense of distributions for $y\in(0,1)$, with some $C_{y_{\ast}}\neq0$. Multiplying \eqref{F6.4} with $\overline{h(y)}$, integrating over $(0,1)$, and taking the imaginary part, we get that
$b''(y_{\ast})h(y_{\ast})=0$. Therefore
\begin{equation}\label{F6.5}
H(y):=\frac{b''(y)h(y)}{b(y)-b(y_{\ast})}\in L^2(\R)\qquad {\rm and} \qquad H\not\equiv0.
\end{equation}
It follows from \eqref{F6.3} that $H$ satisfies 
\begin{equation}\label{F6.6}
(b(y)-b(y_{\ast}))H(y)+b''(y)\int_0^1G_{k_{\ast}}(y,z)H(z)\,dz=0,
\end{equation}
which contradicts our spectral assumption that $b(y_{\ast})$ is not an embedded eigenvalue for $L_k$. The lemma is then proved.
\end{proof}

\subsection{The limiting absorption principles in the variable $v$}
For our purposes, it is more convenient to work with the variables
\begin{equation}\label{F8}
v=b(y), \qquad w=b(y_0), \qquad {\rm for}\,\,y,y_0\in[0,1].
\end{equation}
Set 
\begin{equation}\label{F8.0}
\lv:=b(0),\qquad \uv:=b(1), \qquad B(v):=b'(y)\qquad {\rm and}\qquad \Psi(v):=\varphi(y)\,\,{\rm for}\,\,v\in[\lv,\uv].
\end{equation}
Set also
\begin{equation}\label{F15}
\mathcal{G}_k(v,w):=G_k(y,z),\qquad {\rm where}\qquad v=b(y),\,\,w=b(z).
\end{equation}
$\mathcal{G}_k$ verifies the equation
\begin{equation}\label{F16}
k^2\mathcal{G}_k(v,w)-B^2(v)\partial_v^2\mathcal{G}_k(v,w)-B(v)\partial_vB(v)\partial_v\mathcal{G}_k(v,w)=B(w)\delta(v-w),
\end{equation}
with $\mathcal{G}_k(\lv,w)=\mathcal{G}_k(\uv,w)=0$. We note also that $\mathcal{G}_k(v,w)=\mathcal{G}_k(w,v),$ for $v,w\in[\lv,\uv]$.

\begin{lemma}\label{F16.1}
Set for any $g\in H^1_k(\R)$ and $\epsilon\in[-1/4,1/4]\backslash\{0\}, k\in\mathbb{Z}\backslash\{0\}, w\in[\lv,\uv]$,
\begin{equation}\label{F16.2}
S_{k,w,\epsilon}g(v):=\int_{\R}\Psi(v)\mathcal{G}_k(v,v')\partial_{v'}B(v')\frac{g(v')}{v'-w+i\epsilon}dv'.
\end{equation}
Then for sufficiently small nonzero $\epsilon$, the following bounds hold
\begin{equation}\label{F16.3}
\|S_{k,w,\epsilon}g\|_{H^1_k(\R)} \lesssim|k|^{-1/3} \|g\|_{H^1_k(\R)}\qquad {\rm and}\qquad\|g\|_{H^1_k(\R)}\lesssim\|g+S_{k,w,\epsilon}g\|_{H^1_k(\R)}
\end{equation}
for all $k\in\mathbb{Z}\backslash\{0\},w\in[\lv,\uv]$ and  any $g\in H^1_k(\R)$.
\end{lemma}
\begin{proof}
The lemma follows from Lemma \ref{T2} and Lemma \ref{T5}, in view of the change of the variables formula \eqref{F8}.
\end{proof}

For applications below, we also need a slightly reformulated version of Lemma \ref{F16.1}.
\begin{lemma}\label{F16.4}
Set for any $g\in H^1_k(\R)$ and $\epsilon\in[-1/4,1/4]\backslash\{0\}, k\in\mathbb{Z}\backslash\{0\}, w\in[\lv,\uv]$,
\begin{equation}\label{F16.5}
S'_{k,w,\epsilon}g(v):=\int_{\R}\Psi(v+w)\mathcal{G}_k(v+w,v'+w)\partial_{v'}B(v'+w)\frac{g(v')}{v'+i\epsilon}dv'.
\end{equation}
Then for sufficiently small nonzero $\epsilon$, the following bounds hold
\begin{equation}\label{F16.6}
\|S'_{k,w,\epsilon}g\|_{H^1_k(\R)} \lesssim|k|^{-1/3} \|g\|_{H^1_k(\R)}\qquad {\rm and}\qquad\|g\|_{H^1_k(\R)}\lesssim\|g+S'_{k,w,\epsilon}g\|_{H^1_k(\R)}
\end{equation}
for all $k\in\mathbb{Z}\backslash\{0\},w\in[\lv,\uv]$ and $g\in H^1(\R)$.

\end{lemma}
\begin{proof}
Lemma \ref{F16.4} follows directly from Lemma \ref{F16.1}, by a shift of variables $v\to v-w$. 
\end{proof}

\section{Gevrey bounds for generalized eigenfunctions}\label{gev}
In this section we study the regularity of the generalized eigenfunctions $\psi^{\iota}_{k,\epsilon}(y,y_0), y,y_0\in[0,1], \iota\in\{\pm\}$, see the definitions \eqref{F7} and \eqref{F6}. The main difficulty is the presence of singularity in the generalized eigenfunctions.  The essential idea is to suitably shift the variables so that the resulting functions become  smooth in one of the variables.  The procedure exactly captures the nature of the singular behavior of the generalized eigenfunctions, and is one of our main new observations.

The starting point is the equation \eqref{F7} for the generalized eigenfunctions $\psi^{\iota}_{k,\epsilon}(y,y_0), y,y_0\in[0,1], \iota\in\{\pm\}$, which can be reformulated as
\begin{equation}\label{H0}
\psi^{\iota}_{k,\epsilon}(y,y_0)+\int_0^1G_k(y,z)\frac{b''(z)}{b(z)-b(y_0)+i\iota\epsilon}\psi^{\iota}_{k,\epsilon}(z,y_0)dz=\int_0^1G_k(y,z)\frac{\omega_0^k(z)}{b(z)-b(y_0)+i\iota\epsilon}dz.
\end{equation}
By Lemma \ref{T2} and Lemma \ref{T5}, we have the following bounds for sufficiently small $\epsilon\neq0, y_0\in[0,1],k\in\mathbb{Z}\backslash\{0\}$:
\begin{equation}\label{H0.1}
\left\|\varphi(y)\psi^{\iota}_{k,\epsilon}(y,y_0)\right\|_{H^1_k(\R)}\lesssim |k|^{-1/3}\left\|\omega_0^k\right\|_{H^1_k(\R)}.
\end{equation}
The bounds \eqref{H0.1} are useful to control the low frequency components of the generalized eigenfunctions, and will be an important stepping stone in the treatment of high frequencies.

We begin with an observation on the case when $y_0\in(0,\vartheta_1/2)\cup(1-\vartheta_1/2,1)$.
\begin{lemma}\label{H1}
We have
\begin{equation}\label{F10.4}
\lim_{\epsilon\to0+}\Big[\psi^{-}_{k,\epsilon}(y,y_0)-\psi^{+}_{k,\epsilon}(y,y_0)\Big]\equiv 0, \qquad{\rm for}\,\,y_0\in[0,\vartheta_1/2]\cup[1-\vartheta_1/2,1].
\end{equation}
\end{lemma}

\begin{proof}
We use the equation \eqref{F7}. Set for $y,y_0\in[0,1], k\in\mathbb{Z}\backslash\{0\}$,
\begin{equation}\label{H1.1}
h_k(y,y_0):=\lim_{\epsilon\to0+}\Big[\psi^{-}_{k,\epsilon}(y,y_0)-\psi^{+}_{k,\epsilon}(y,y_0)\Big].
\end{equation}
Note that for $y_0\in[0,\vartheta_1/2]\cup[1-\vartheta_1/2,1]$, $y_0\not\in{\rm supp}\,\omega_0^k\cup{\rm supp}\,b''$, which is contained in $[\vartheta_1,1-\vartheta_1]$. Therefore for $y_0\in(0,\vartheta_1/2)\cup(1-\vartheta_1/2,1)$, $h_k(y,y_0)$ satisfies
\begin{equation}\label{H2}
-k^2h_k(y,y_0)+\frac{d^2}{dy^2}h_k(y,y_0)-\frac{b''(y)}{b(y)-b(y_0)}h_k(y,y_0)=0,\qquad{\rm for}\,\,y\in(0,1),
\end{equation}
with $h_k(0,y_0)=h_k(1,y_0)=0$. Now set
\begin{equation}\label{H3}
g_k(y,y_0):=\frac{b''(y)}{b(y)-b(y_0)}h_k(y,y_0),\qquad {\rm for}\,\,y_0\in[0,\vartheta_1/2]\cup[1-\vartheta_1/2,1].
\end{equation}
It is clear that $g_k(\cdot,y_0)\in L^2(\R)$ for $y_0\in[0,\vartheta_1/2]\cup[1-\vartheta_1/2,1]$. It follows from \eqref{H2} that
\begin{equation}\label{H4}
(b(y)-b(y_0))g_k(y,y_0)+b''(y)\int_0^1G_k(y,z)g_k(z,y_0)\,dz=0,\qquad {\rm for}\,\,y\in[0,1],
\end{equation}
which implies that $g_k(y,y_0)\equiv 0$ for $y_0\in [0,\vartheta_1/2]\cup[1-\vartheta_1/2,1]$, in view of the spectral assumption that $b(y_0)$ is not an embedded eigenvalue for $L_k$. Then \eqref{H2} implies that $h_k(y,y_0)\equiv0$ for $y_0\in [0,\vartheta_1/2]\cup[1-\vartheta_1/2,1]$. The lemma is  proved.
\end{proof}


We now turn to the main case when $y_0\in[\vartheta_1/2,1-\vartheta_1/2]$.
Recall the change of variables \eqref{F8}, and set for $k\in\mathbb{Z}\backslash\{0\}, \iota\in\{\pm\}, y,y_0\in[0,1]$ and sufficiently small $\epsilon\neq0$,
\begin{equation}\label{F9}
\phi^{\iota}_{k,\epsilon}(v,w):=\psi^{\iota}_{k,\epsilon}(y,y_0),\qquad f_0^k(v):=\omega_0^k(y).
\end{equation}

We need the following structural bound on the Green's function $\mathcal{G}_k$.
\begin{lemma}\label{F18}
Define the localized Green's function $\mathcal{G}_k^{\ast}$ as
\begin{equation}\label{F19}
\mathcal{G}^{\ast}_k(v,w):=\Psi(v)\mathcal{G}_k(v,w)\Psi(w),\qquad {\rm for}\,\,v,w\in\R. 
\end{equation}
Then for some $\delta_0:=\delta_0(\vartheta_1,\lambda,s)>0$, we have the bounds
\begin{equation}\label{F20}
\big|\widetilde{\mathcal{G}_k^{\ast}}(\xi,\eta)\big|\lesssim \frac{e^{-\delta_0\langle\xi+\eta\rangle^{(s+1)/2}}}{k^2+|\eta|^2},\qquad{\rm for}\,\,\xi,\eta\in\R.
\end{equation}
\end{lemma}
We postpone the proof to subsection \ref{Gr} in the appendix, and proceed to present our main argument.\\

The following lemma contains the main estimates for the generalized eigenfunctions.
\begin{lemma}\label{S1}
Define for $\iota\in\{\pm\}, k\in\mathbb{Z}\backslash\{0\}$ and sufficiently small $\epsilon>0$ (so that Lemma \ref{F16.4} holds),
\begin{equation}\label{F17.1}
\Theta^{\iota}_{k,\epsilon}(v,w):=\Psi(v+w)\phi^{\iota}_{k,\epsilon}(v+w,w)\Psi(w), \qquad {\rm for}\,\,v,w\in\R.
\end{equation}

We have the following bounds
\begin{equation}\label{S2}
\left\|(|k|+|\xi|)e^{\lambda\langle k,\eta\rangle^s}\,\widetilde{\,\,\Theta^{\iota}_{k,\epsilon}}(\xi,\eta)\right\|_{L^2(\xi\in\R,\eta\in\R)}\lesssim \left\|e^{\lambda\langle k,\eta\rangle^s}\,\widetilde{f_0^k}(\eta)\right\|_{L^2(\eta\in\R)},
\end{equation}
and 
\begin{equation}\label{S3}
\sup_{\xi,\eta\in \R}\left|(|k|+|\xi|)^2e^{\lambda\langle k,\eta\rangle^s}\,\widetilde{\,\,\Theta^{\iota}_{k,\epsilon}}(\xi,\eta)\right|\lesssim \left\|e^{\lambda\langle k,\eta\rangle^s}\,\widetilde{f_0^k}(\eta)\right\|_{L^2(\eta\in\R)}.
\end{equation}
\end{lemma}

\begin{proof}
We first note, using \eqref{H0.1} and the definitions \eqref{F17.1}, the following bounds 
\begin{equation}\label{F17.3}
\left\|(|k|+|\partial_v|)\Theta^{\iota}_{k,\epsilon}\right\|_{L^2_vL^2_w}\lesssim \langle k\rangle \left\|f_0^k\right\|_{H^3},
\end{equation}
which is useful to control the low frequency components of $\Theta^{\iota}_{k,\epsilon}$.

We first present the proof of \eqref{S2}-\eqref{S3} under the assumption that the left hand side of \eqref{S2} is finite, for the sake of clarity of the main ideas. We shall indicate in the end of the section how to remove this qualitative assumption by using approximate weights to $A_k$. 

We divide the rest of the proof into several steps.

{\bf Step 1} In this step we derive the main equations for $\Theta^{\iota}_{k,\epsilon}(v,w)$.
From the definitions \eqref{F9} and the equations \eqref{F7}, it follows that $\phi^{\iota}_{k,\epsilon}(v,w)$ satisfies for $v,w\in[\lv,\uv]$,
\begin{equation}\label{F10}
\begin{split}
&k^2\phi^{\iota}_{k,\epsilon}(v,w)-B^2(v)\partial_v^2\phi^{\iota}_{k,\epsilon}(v,w)-B(v)\partial_vB(v)\partial_v\phi^{\iota}_{\epsilon}(v,w)\\
&+B(v)\partial_vB(v)\frac{\phi^{\iota}_{k,\epsilon}(v,w)}{v-w+i\iota\epsilon}=\frac{f_0^k(v)}{v-w+i\iota \epsilon}.
\end{split}
\end{equation}

Using the definitions of $\mathcal{G}_k$, see \eqref{F15}-\eqref{F16}, with localization in $v$, we can reformulate equation \eqref{F10} as
\begin{equation}\label{F17}
\begin{split}
&\Psi(v)\phi^{\iota}_{k,\epsilon}(v,w)+\int_{\mathbb{R}}\Psi(v)\mathcal{G}_k(v,v')\partial_{v'}B(v')\frac{\Psi(v')\phi^{\iota}_{k,\epsilon}(v',w)}{v'-w+i\iota\epsilon}\,dv'\\
&=\int_{\mathbb{R}}\Psi(v)\mathcal{G}_k(v,v')\frac{1}{B(v')}\frac{f_0^k(v')}{v'-w+i\iota\epsilon}\,dv',\qquad{\rm for}\,\,v\in\R, w\in[\lv,\uv].
\end{split}
\end{equation}
In the above we used the fact that $\Psi\equiv 1$ on the support of $\partial_vB$.
Hence $\Theta^{\iota}_{k,\epsilon}(v,w)$ satisfies the more regular (in $w$) equation
\begin{equation}\label{F17.2}
\begin{split}
&\Theta^{\iota}_{k,\epsilon}(v,w)+\int_{\mathbb{R}}\Psi(v+w)\mathcal{G}_k(v+w,v'+w)\partial_{v'}B(v'+w)\frac{\Theta^{\iota}_{k,\epsilon}(v',w)}{v'+i\iota\epsilon}\,dv'\\
&=\int_{\mathbb{R}}\Psi(v+w)\mathcal{G}_k(v+w,v'+w)\Psi(w)\frac{1}{B(v'+w)}\frac{f_0^k(v'+w)}{v'+i\iota\epsilon}\,dv',\qquad{\rm for}\,\,v,w\in\R.
\end{split}
\end{equation}

{\bf Step 2} We now study the regularity of $\Theta^{\iota}_{k,\epsilon}$ using equation \eqref{F17.2}. Define the Fourier multiplier $A_k$ as
\begin{equation}\label{F24}
\widetilde{\,\,A_kh\,\,}(\eta):=e^{\lambda\langle k,\eta\rangle^s}\,\widetilde{\,\,h\,\,}(\eta),\qquad{\rm for\,\,any\,\,}h\in L^2(\R).
\end{equation}
The basic idea is to use the limiting absorption principle, see Lemma \ref{F16.4}, to bound $\Theta^{\iota}_{k,\epsilon}$. We note that $\Theta^{\iota}_{k,\epsilon}$ is very smooth in $w$ but not so in $v$, due to the presence of the singular factor $1/(v+i\iota\epsilon)$. In order to prove Gevrey regularity of $\Theta^{\iota}_{k,\epsilon}$ in $w$, we apply the operator $A_k$, which acts on the variable $w$, to equation \eqref{F17.2} and obtain
\begin{equation}\label{F25}
\begin{split}
&A_k\Theta^{\iota}_{k,\epsilon}(v,w)+\int_{\mathbb{R}}\mathcal{G}^{\ast}_k(v+w,v'+w)\partial_{v'}B(v'+w)\frac{A_k\Theta^{\iota}_{k,\epsilon}(v',w)}{v'+i\iota\epsilon}\,dv'\\
&=A_k\left[\int_{\mathbb{R}}\mathcal{G}^{\ast}_k(v+\cdot,v'+\cdot)\Psi(\cdot)\frac{1}{B(v'+\cdot)}\frac{f_0^k(v'+\cdot)}{v'+i\iota\epsilon}\,dv'\right](w)+\mathcal{C}^{\iota}_{k,\epsilon}(v,w)\\
&:=F^{\iota}_{k,\epsilon}(v,w)+\mathcal{C}^{\iota}_{k,\epsilon}(v,w),
\end{split}
\end{equation}
for $v,w\in\R$, where the commutator term $\mathcal{C}^{\iota}_{k,\epsilon}(v,w)$ is defined as
\begin{equation}\label{S26}
\begin{split}
\mathcal{C}^{\iota}_{k,\epsilon}(v,w)&:=\int_{\mathbb{R}}\mathcal{G}^{\ast}_k(v+w,v'+w)\partial_{v'}B(v'+w)\frac{A_k\Theta^{\iota}_{k,\epsilon}(v',w)}{v'+i\iota\epsilon}\,dv'\\
&-A_k\left[\int_{\mathbb{R}}\mathcal{G}^{\ast}_k(v+\cdot,v'+\cdot)\partial_{v'}B(v'+\cdot)\frac{\Theta^{\iota}_{k,\epsilon}(v',\cdot)}{v'+i\iota\epsilon}\,dv'\right](w).
\end{split}
\end{equation}

Now fix a smooth cutoff function $\Psi_0$ with ${\rm supp}\,\Psi_0\subseteq (\lv,\uv)$, $\Psi_0\equiv 1$ on the support of $\Psi$, and 
\begin{equation}\label{F25.1}
\left|\widetilde{\,\Psi_0}(\xi)\right|\lesssim e^{-\langle \xi\rangle^{(1+s)/2}},\qquad {\rm for}\,\,\xi\in\R.
\end{equation}

Applying Lemma \ref{F16.4} for each $w$ and taking $L^2$ in $w$, we obtain from \eqref{F25} 
\begin{equation}\label{F26.1}
\left\|\Psi_0(w)(|k|+|\partial_v|)A_k\Theta^{\iota}_{k,\epsilon}(v,w)\right\|_{L^2_vL^2_w}\lesssim \left\|(|k|+|\partial_v|)F^{\iota}_{k,\epsilon}\right\|_{L^2_vL^2_w}+\left\|(|k|+|\partial_v|)\mathcal{C}^{\iota}_{k,\epsilon}\right\|_{L^2_vL^2_w}.
\end{equation}


{\bf Step 3} In this step we bound the terms on the right hand side of \eqref{F26.1}, and estimate $A_k\Theta^{\iota}_{k,\epsilon}(v,w)$ using $\Psi_0(w)A_k\Theta^{\iota}_{k,\epsilon}(v,w)$. More precisely,  we claim that
\begin{claim}\label{claim1}
The term $F^{\iota}_{k,\epsilon}(v,w)$ satisfies the bounds
\begin{equation}\label{F28}
\left\|(|k|+|\partial_v|)F^{\iota}_{k,\epsilon}(v,w)\right\|_{L^2(\R\times\R)}\lesssim \left\|A_kf_0^k\right\|_{L^2(\R)},
\end{equation}
\begin{equation}\label{F28.1}
\sup_{\xi,\eta\in\R}\left[(|k|^2+|\xi|^2)A_k(\eta)\big|\widetilde{F^{\iota}_{k,\epsilon}}(\xi,\eta)\big|\right]\lesssim \left\|A_kf_0^k\right\|_{L^2(\R)}.
\end{equation}
\end{claim}

\begin{claim}\label{claim2}
The term $\mathcal{C}^{\iota}_{k,\epsilon}(v,w)$ satisfies the bounds for any $\beta\in(0,1)$
\begin{equation}\label{F29}
\begin{split}
\|(|k|+|\partial_v|)\mathcal{C}^{\iota}_{k,\epsilon}(v,w)\|_{L^2_vL^2_w(\R\times\R)}\lesssim& \beta\big\|(|k|+|\partial_v|)A_k\Theta^{\iota}_{k,\epsilon}(v,w)\big\|_{L^2_vL^2_w(\R\times\R)}\\
                         &+C_{\beta} \big\|(|k|+|\partial_v|)\Theta^{\iota}_{k,\epsilon}(v,w)\big\|_{L^2_vL^2_w(\R\times\R)}.
\end{split}
\end{equation}
\end{claim}

\begin{claim}\label{claim3}
We have the following bounds
\begin{equation}\label{F100}
\left\|(|k|+|\partial_v|)A_k\Theta^{\iota}_{k,\epsilon}(v,w)\right\|_{L^2_vL^2_w}\lesssim \left\|\Psi_0(w)(|k|+|\partial_v|)A_k\Theta^{\iota}_{k,\epsilon}(v,w)\right\|_{L^2_vL^2_w}+\left\|(|k|+|\partial_v|)\Theta^{\iota}_{k,\epsilon}\right\|_{L^2_vL^2_w}.
\end{equation}
\end{claim}

We postpone the proofs of Claim \ref{claim1} to Claim \ref{claim3} to the end of the section.

{\bf Step 4} We now complete the proof of \eqref{S2} using the bounds \eqref{F26.1}-\eqref{F100}. Indeed, we obtain from \eqref{F26.1}-\eqref{F100} that
\begin{equation}\label{F37}
\begin{split}
\left\|(|k|+|\partial_v|)A_k\Theta^{\iota}_{k,\epsilon}(v,w)\right\|_{L^2_vL^2_w}\lesssim& \left\|A_kf_0^k\right\|_{L^2(\R)}+ \beta\Big\|(|k|+|\partial_v|)A_k\Theta^{\iota}_{k,\epsilon}(v,w)\Big\|_{L^2_vL^2_w}\\
                         &+C_{\beta} \Big\|(|k|+|\partial_v|)\Theta^{\iota}_{k,\epsilon}(v,w)\Big\|_{L^2_vL^2_w(\R\times\R)}.
\end{split}
\end{equation}
Choose $\beta\in(0,1)$ sufficiently small, and use the low frequency bounds \eqref{F17.3}, \eqref{S2} then follows. 

{\bf Step 5} We now prove the bounds \eqref{S3}. In view of \eqref{F17.2} and \eqref{F28.1}, \eqref{S3} follows from 
\begin{equation}\label{F37.1}
\begin{split}
&(k^2+\xi^2)A_k(\eta)\left|\int_{\R^3}\widetilde{\,\mathcal{G}_k^{\ast}}(\xi,\zeta) \widetilde{\partial_vB}(\alpha) \widetilde{\,\,\Theta^{\iota}_{k,\epsilon}}(-\zeta-\gamma-\alpha,\eta-\xi-\zeta-\alpha)\,d\alpha d\zeta d\gamma\right|\\
&\lesssim A_k(\eta)\int_{\R^3}e^{-0.6\delta_0\langle \xi+\zeta,\alpha\rangle^{(1+s)/2}}\left|\widetilde{\,\,\Theta^{\iota}_{k,\epsilon}}(-\zeta-\gamma-\alpha,\eta-\xi-\zeta-\alpha)\right|\,d\alpha d\zeta d\gamma\\
&\lesssim \int_{\R^3}e^{-0.5\delta_0\langle \xi+\zeta,\alpha\rangle^{(1+s)/2}}A_k(\eta-\xi-\zeta-\alpha)\left|\widetilde{\,\,\Theta^{\iota}_{k,\epsilon}}(-\zeta-\gamma-\alpha,\eta-\xi-\zeta-\alpha)\right|\,d\alpha d\zeta d\gamma\\
&\lesssim \|(|k|+|\partial_v|)A_k\Theta^{\iota}_{k,\epsilon}\|_{L^2_vL^2_w}\lesssim\left\|A_k(\eta)\widetilde{f_0^k}(\eta)\right\|_{L^2(\eta\in\R)}.
\end{split}
\end{equation}
In the above, $\delta_0:=\delta_0(\vartheta_1)>0$, and we used the elementary point-wise inequality \eqref{F31.4}. See \eqref{F34}-\eqref{F35} for related computations.
 The proof of Lemma \ref{S1} is now complete.
\end{proof}

We now present the proof of Claim \ref{claim1} through Claim \ref{claim3}. 

We use the following elementary inequalities: if $a, b\in\R^n$ and $\beta\in(0,1)$ then
\begin{equation}\label{b>a}
\langle b\rangle\ge \beta\langle a-b\rangle\qquad {\rm implies}\qquad \langle a\rangle^{s}\leq \langle b\rangle^{s}+(1-\mu)\langle a-b\rangle^{s},
\end{equation}
for some $\mu:=\mu(\beta,s)>0$, and the point-wise bounds for any $\zeta,\eta\in\R$:
\begin{equation}\label{P3.2}
\left|e^{\lambda\langle k,\eta\rangle^s}-e^{\lambda\langle k,\eta-\zeta\rangle^s}\right|\lesssim_{\lambda,s}\langle k, \eta-\zeta\rangle^{s-1}e^{\lambda(\langle k,\eta-\zeta\rangle^s+\langle \zeta\rangle^s)}.
\end{equation}
\eqref{P3.2} can be proved by considering the cases $\langle\zeta\rangle\ge\langle k, \eta-\zeta\rangle/4$ and $\langle\zeta\rangle\leq\langle k,\eta-\zeta\rangle/4$, using also \eqref{b>a}.

\begin{proof}[Proof of Claim \ref{claim1}]
Using \eqref{F25} and taking Fourier transform in $v,w$, we obtain that
\begin{equation}\label{F30}
\begin{split}
\widetilde{F^{\iota}_{k,\epsilon}}(\xi,\eta)&=CA_k(\eta)\int_{\R^4}\widetilde{\mathcal{G}_k^{\ast}}(\xi,\zeta)e^{-iw\eta+i\xi w+i\zeta(v'+w)}\Psi(w)\frac{f_0^k(v'+w)}{B(v'+w)} e^{i(v'+i\iota\epsilon)\gamma}\mathbf{1}_{\iota\gamma>0}\,dv'dw d\zeta d\gamma\\
&=CA_k(\eta)\int_{\R^2}\widetilde{\mathcal{G}_k^{\ast}}(\xi,\zeta)\widetilde{h_k}(-\zeta-\gamma,\eta-\xi-\zeta)e^{-\epsilon\iota\gamma}\mathbf{1}_{\iota\gamma>0}\,d\zeta d\gamma.
\end{split}
\end{equation}
In the above, we have set
\begin{equation}\label{F31}
h_k(v,w):=\Psi(w)f_0^k(v+w)/B(v+w),\qquad {\rm for}\,\,v,w\in\R.
\end{equation}
Since $\Psi\equiv 1$ on the support of $f_0^k$, we can write
\begin{equation}\label{F31.1}
h_k(v,w)=\Upsilon(v,w)f_0^k(v+w), \qquad{\rm where}\,\,\Upsilon(v,w):=\Psi(w)\Psi(v+w)/B(v+w).
\end{equation}
Using general properties of Gevrey spaces, see Lemma \ref{lm:Gevrey} and Lemma \ref{GPF}, and the regularity of $b$, see \eqref{B1}, we obtain that for some $\delta_0=\delta_0(\vartheta_1)>0$ (with a slight abuse of notation, see \eqref{F20}),
\begin{equation}\label{F31.2}
\big|\widetilde{\,\Upsilon\,}(\xi,\eta)\big|\lesssim e^{-\delta_0\langle\xi,\eta\rangle^{(s+1)/2}},\qquad {\rm for}\,\,\xi,\eta\in\R.
\end{equation}
Therefore,
\begin{equation}\label{F31.3}
\big|\widetilde{\,h_k}(\xi,\eta)\big|\lesssim \int_{\R}e^{-\delta_0\langle \xi-\alpha,\eta-\alpha\rangle^{(1+s)/2}}\widetilde{\,f_0^k}(\alpha)\,d\alpha.
\end{equation}
Using the elementary inequality (which follows from \eqref{b>a})
\begin{equation}\label{F31.4}
A_k(\eta)\lesssim A_k(\eta-\alpha)e^{\lambda\langle\alpha\rangle^s},\qquad {\rm for}\,\,\alpha,\eta\in\R,
\end{equation}
we obtain that
\begin{equation}\label{F31.5}
A_k(\eta)\big|\widetilde{\,h_k}(\xi,\eta)\big|\lesssim  \int_{\R}e^{-0.5\delta_0\langle \xi-\alpha,\eta-\alpha\rangle^{(1+s)/2}}A_k(\alpha)\widetilde{\,f_0^k}(\alpha)\,d\alpha.
\end{equation}
Thus in view of \eqref{F20} and \eqref{F30}, \eqref{F31.5} implies that
\begin{equation}\label{F32}
\begin{split}
&(k^2+\xi^2)\big|\widetilde{F^{\iota}_{k,\epsilon}}(\xi,\eta)\big|\lesssim \int_{\R^2}e^{-0.6\delta_0\langle \xi+\zeta\rangle^{(s+1)/2}}A_k(\eta-\xi-\zeta)\left|\widetilde{h_k}(-\zeta-\gamma,\eta-\xi-\zeta)\right|d\zeta d\gamma\\
&\lesssim \int_{\R^2}e^{-0.1\delta_0\langle \xi+\zeta,\eta-\alpha\rangle^{(s+1)/2}}A_k(\alpha)\big|f_0^k(\alpha)\big|\,d\zeta d\alpha\lesssim \int_{\R}e^{-0.1\delta_0\langle \eta-\alpha\rangle^{(s+1)/2}}A_k(\alpha)\big|\widetilde{f_0^k}(\alpha)\big|\, d\alpha.
\end{split}
\end{equation}
The desired bounds \eqref{F28} and \eqref{F28.1} then follow from \eqref{F32}. 

\end{proof}

\begin{proof}[Proof of Claim \ref{claim2}]
 It follows from \eqref{S26} that
\begin{equation}\label{F34}
\begin{split}
&\widetilde{\mathcal{C}^{\iota}_{k,\epsilon}}(\xi,\eta)\\
&:=C\bigg[\int_{\R^4}\widetilde{\mathcal{G}_k^{\ast}}(\xi,\zeta)e^{-i\eta w+i\xi w+i\zeta(v'+w)}\partial_{v'}B(v'+w)A_k\Theta^{\iota}_{k,\epsilon}(v',w)e^{i(v'+i\iota\epsilon)\gamma}\mathbf{1}_{\iota\gamma>0}\, dv' dw d\zeta d\gamma\\
&-A_k(\eta)\int_{\R^4}\widetilde{\mathcal{G}^{\ast}_k}(\xi,\zeta)e^{-i\eta w+i\xi w+i\zeta(v'+w)}\partial_{v'}B(v'+w)\Theta^{\iota}_{k,\epsilon}(v',w)e^{i(v'+i\iota\epsilon)\gamma}\mathbf{1}_{\iota\gamma>0}\, dv' dw d\zeta d\gamma\bigg]\\
&=C\int_{\R^3}\widetilde{\mathcal{G}^{\ast}_k}(\xi,\zeta)\widetilde{\partial_{v}B}(\alpha)\bigg[A_k(\eta-\xi-\zeta-\alpha)-A_k(\eta)\bigg]\widetilde{\,\,\Theta^{\iota}_{k,\epsilon}}(-\zeta-\gamma-\alpha,\eta-\xi-\zeta-\alpha)\,d\alpha d\zeta d\gamma.
\end{split}
\end{equation}
Thus by \eqref{F20} and \eqref{P3.2},
\begin{equation}\label{F35}
\begin{split}
&(k^2+\xi^2)\left|\widetilde{\mathcal{C}^{\iota}_{k,\epsilon}}(\xi,\eta)\right|\lesssim \int_{\R^3} e^{-0.6\delta_0\langle\xi+\zeta\rangle^{(s+1)/2}}e^{\lambda\langle\alpha\rangle^s}\left|\widetilde{\partial_{v}B}(\alpha)\right|\langle \eta-\xi-\zeta-\alpha \rangle^{s-1}\\
&\hspace{1in} \times A_k(\eta-\xi-\zeta-\alpha)\left|\widetilde{\,\,\Theta^{\iota}_{k,\epsilon}}(-\zeta-\gamma-\alpha,\eta-\xi-\zeta-\alpha)\right|\,d\alpha d\zeta d\gamma\\
&\lesssim \int_{\R^3}e^{-0.5\delta_0\langle \alpha,\zeta\rangle^{(1+s)/2}}\langle\eta-\zeta-\alpha\rangle^{s-1} A_k(\eta-\zeta-\alpha)\left|\widetilde{\,\,\Theta^{\iota}_{k,\epsilon}}(\gamma,\eta-\zeta-\alpha)\right|\,d\alpha d\zeta d\gamma.
 \end{split}
\end{equation}
As a consequence, 
\begin{equation}\label{F36}
\left\|(|k|+|\partial_v|)\mathcal{C}^{\iota}_{k,\epsilon}(v,w)\right\|_{L^2_vL^2_{w}}\lesssim \left\|(|k|+|\partial_v|)\langle\partial_w\rangle^{s-1} A_k\Theta^{\iota}_{k,\epsilon}\right\|_{L^2_vL^2_w},
\end{equation}
from which \eqref{F29} follows.

\end{proof}

\begin{proof}[Proof of Claim \ref{claim3}]
Define for $v,w\in\R$,
\begin{equation}\label{C32}
H(v,w):=(|k|+|\partial_v|)A_k\Theta^{\iota}_{k,\epsilon}(v,w)-(|k|+|\partial_v|)\Psi_0(w)A_k\Theta^{\iota}_{k,\epsilon}(v,w).
\end{equation}
For the simplicity of notations we suppressed the dependence of $H$ on $\iota,\epsilon,k$ in the above definition. By the support property of $\Psi_0$ and the inequality \eqref{P3.2} we have
\begin{equation}\label{P4}
\begin{split}
\left|\widetilde{\,H\,}(\xi,\eta)\right|&=(|k|+|\xi|)\left|\int_{\R}\widetilde{\,\,\Psi_0}(\zeta)\,\widetilde{\,\,\Theta^{\iota}_{k,\epsilon}}(\xi,\eta-\zeta)\Big[A_k(\eta)-A_k(\eta-\zeta)\Big]\,d\zeta\right|\\
&\lesssim \left|\int_{\R}e^{-0.5\delta_0\langle\zeta\rangle^{(1+s)/2}}\,
\langle\eta-\zeta\rangle^{s-1}(|k|+|\xi|)A_k(\eta-\zeta)\widetilde{\,\,\Theta^{\iota}_{k,\epsilon}}(\xi,\eta-\zeta)\,d\zeta\right|.
\end{split}
\end{equation}
Therefore,
\begin{equation}\label{C33}
\begin{split}
\left\|\widetilde{\,H\,}(\xi,\eta)\right\|_{L^2_{\xi}L^2_{\eta}}&\lesssim \left\|\langle\partial_w\rangle^{s-1}(|k|+|\partial_v|)A_k\Theta^{\iota}_{k,\epsilon}(v,w)\right\|_{L^2_vL^2_w}\\
                                                                               &\lesssim \beta \left\|(|k|+|\partial_v|)A_k\Theta^{\iota}_{k,\epsilon}(v,w)\right\|_{L^2_vL^2_w}+C_{\beta}\left\|(|k|+|\partial_v|)\Theta^{\iota}_{k,\epsilon}(v,w)\right\|_{L^2_vL^2_w},
\end{split}
\end{equation}
for any $\beta>0$. \eqref{F100} follows from \eqref{C33} and the definition \eqref{C32}, upon choosing a sufficiently small $\beta>0$.

\end{proof}

We now say a few words on how to remove the qualitative assumption that $$\|(|k|+|\partial_v|)A_k\widetilde{\,\,\Theta^{\iota}_{k,\epsilon}}\|_{L^2_vL^2_w}<\infty.$$ The argument is standard. One can for example follow the technique in the appendix of \cite{IOJI} and introduce for $\rho\gg1$,
\begin{equation}\label{P1}
h_{\rho}(r):=\left\{\begin{array}{ll}
sr^{s-1}-s\rho^{s-1}& {\rm if} \,\,r\in(0,\rho]\\
0&{\rm if}\,\,r\ge \rho,
\end{array}\right.
\end{equation}
\begin{equation}\label{P1.1}
\Pi_{\rho}(r):=\int_0^rh_{\rho}(x)\,dx,
\end{equation}
and define 
\begin{equation}\label{P2}
A_k^{\rho}(\eta):=e^{\lambda \Pi_{\rho}(\langle k,\eta\rangle))}.
\end{equation}
Clearly $A_k^{\rho}(\eta)$ is a bounded function (with a bound that depends on $\rho>1$), $A_k^{\rho}(\eta)\leq A_k(\eta)$ and $A_k^{\rho}(\eta)\uparrow A_k(\eta)$ as $\rho\to\infty$ for any $\eta\in\R$. The idea is to use $A_k^{\rho}(\eta)$ in the proof of \eqref{S2} and then send $\rho\to\infty$. We only need to use the following properties of $A_k^{\rho}$ instead of the point-wise inequalities \eqref{P3.2} and \eqref{F31.4}:

\begin{equation}\label{P3}
A_k^{\rho}(\alpha+\beta)\lesssim A_k^{\rho}(\alpha)\,e^{\lambda\langle\beta\rangle^s},
\end{equation}
\begin{equation}\label{P4}
\left|A_k^{\rho}(\alpha)-A^{\rho}_k(\beta)\right|\lesssim  \langle k,\alpha\rangle^{s-1} A_k^{\rho}(\alpha)\,e^{\lambda \langle \alpha-\beta\rangle^s}.
\end{equation}
for any $\alpha,\beta\in\R$, where the implied constants are independent of $\rho\gg1$. The elementary inequalities \eqref{P3}-\eqref{P4} can be proved from the definitions, using the fact that $h_{\rho}(r)\leq sr^{s-1}$ and \eqref{b>a}. We omit the routine details.

\section{Proof of the main theorem}\label{evo}

In this section we complete the proof of Theorem \ref{thm}.
\begin{proof}[Proof of Theorem \ref{thm}]
We first derive the properties of the stream function $\phi_k(t,v)$.
Set 
\begin{equation}\label{D1.1}
\Theta_k(v,w):=\lim_{\epsilon\to0+}\left[\Theta^{-}_{k,\epsilon}(v,w)-\Theta^{+}_{k,\epsilon}(v,w)\right].
\end{equation}
See Lemma 4.3 in \cite{JiaL} for the existence of the above limit. 

We claim that 
\begin{equation}\label{S4}
\left\|(|k|+|\xi|)e^{\lambda\langle k,\eta\rangle^s}\,\widetilde{\,\,\Theta_{k}}(\xi,\eta)\right\|_{L^2_vL^2_w}+\sup_{\xi,\eta\in\R}\left|(k^2+\xi^2)e^{\lambda\langle k,\eta\rangle^s}\,\widetilde{\,\,\Theta_{k}}(\xi,\eta)\right|\lesssim \left\|e^{\lambda\langle k,\eta\rangle^s}\,\widetilde{f_0^k}(\eta)\right\|_{L^2(\eta\in\R)},
\end{equation}
and denoting $\phi_k(t,v)$ as the $k-$th Fourier coefficient in $z$ of $\phi(t,z,v)$,
\begin{equation}\label{D3}
\widetilde{\,\,\Psi \phi_k}(t,\xi)=C\widetilde{\,\,\Theta_k}(\xi-kt,\xi).
\end{equation}
\eqref{U3} follows from \eqref{S4} and \eqref{D3}. \eqref{S4} follows from \eqref{S2} and \eqref{S3}. 

To prove \eqref{D3}, we recall the change of variables \eqref{F9} and Lemma \ref{H1}.

By \eqref{F3.4}, using the change of variable $v=b(y), w=b(y_0)$, and in view of \eqref{F10.4} and \eqref{F17.1}, we obtain
\begin{equation}\label{D2}
\begin{split}
&\Psi(v)\phi_k(t,v)e^{-iktv}=-\frac{1}{2\pi i}\lim_{\epsilon\to0+}\int_{\R}e^{-ikwt}\Psi(v)\left[\phi^{-}_{k,\epsilon}(v,w)-\phi^{+}_{k,\epsilon}(v,w)\right]\Psi(w)\,dw\\
                &=-\frac{1}{2\pi i}\lim_{\epsilon\to0+}\int_{\R}e^{-ikwt}\left[\Theta^{-}_{k,\epsilon}(v-w,w)-\Theta^{+}_{k,\epsilon}(v-w,w)\right]\,dw\\
                &=-\frac{1}{2\pi i}\int_{\R}e^{-ikwt}\Theta_k(v-w,w)\,dw=C\int_{\R^3}e^{-ikwt}\,\widetilde{\,\,\Theta_k}(\alpha,\beta)e^{i\alpha (v-w)+i\beta w}\,d\alpha d\beta dw\\
                &=C\int_{\R}\widetilde{\,\,\Theta_k}(\alpha,\alpha+kt)e^{i\alpha v}\,d\alpha.
\end{split}
\end{equation}
Hence 
\begin{equation}
\widetilde{\,\,\Psi \phi_k}(t,\xi)=C\widetilde{\,\,\Theta_k}(\xi-kt,\xi),
\end{equation}
which is exactly \eqref{D3}. 

We now turn to the property of $f$. Taking the Fourier transform of $f$ in $z$ and denoting the Fourier coefficients as $f_k(t,v)$, in view of the equation \eqref{F3.0} and the definitions \eqref{U1}-\eqref{U2}, we obtain that
\begin{equation}\label{D5}
\omega_k(t,y)=f_k(t,v(y))\,e^{-iktv(y)},
\end{equation}
and 
\begin{equation}\label{D6}
\partial_tf_k=ikB(v)\partial_vB(v)\phi_k(t,v),\qquad {\rm for}\,\,v\in[\lv,\uv].
\end{equation}
We notice that $\Psi\equiv1$ on the support of $\partial_vB$. Therefore,
\begin{equation}\label{D7}
f_k(t,v)-f_0^k(v)=ik\int_0^tB(v)\partial_vB(v)\Psi(v)\phi_k(\tau,v)\,d\tau.
\end{equation}
Using \eqref{D3}, we obtain
\begin{equation}\label{D8}
\begin{split}
\widetilde{\,f\,\,}(t,k,\xi)-\widetilde{\,f_0}(k,\xi)&=Cik\int_0^t\int_{\R} \widetilde{[B\partial_vB]}(\xi-\zeta) \widetilde{\,\,\Theta_k}(\zeta-k\tau,\zeta)\,d\tau d\zeta.
\end{split}
\end{equation}
Thus,
\begin{equation}\label{D9}
\begin{split}
A_k(\xi)\big|\widetilde{\,f\,\,}(t,k,\xi)-\widetilde{\,f_0}(k,\xi)\big|&\lesssim |k|\int_0^{\infty}\int_{\R}e^{-0.5\delta_0\langle\xi-\zeta\rangle^{(1+s)/2}} A_k(\zeta) \left|\widetilde{\,\,\Theta_k}(\zeta-k\tau,\zeta)\right|\,d\tau d\zeta\\
&\lesssim \int_{\R}\int_{\R}e^{-0.5\delta_0\langle\xi-\zeta\rangle^{(1+s)/2}} A_k(\zeta) \left|\widetilde{\,\,\Theta_k}(\gamma,\zeta)\right|\,d\gamma d\zeta\\
&\lesssim \left\{\int_{\R}\int_{\R}e^{-0.5\delta_0\langle\xi-\zeta\rangle^{(1+s)/2}} A^2_k(\zeta)\,(|k|+|\gamma|)^2 \left|\widetilde{\,\,\Theta_k}(\gamma,\zeta)\right|^2\,d\gamma d\zeta\right\}^{1/2}.
\end{split}
\end{equation}
\eqref{U4} follows from \eqref{D9}. 

To prove \eqref{U5}, we notice that the existence of $\lim_{t\to\infty}f(t)$ is clear, in view of the calculations \eqref{D8}-\eqref{D9} and similarly, we have for any $t'>t>0$
\begin{equation}\label{D10}
\begin{split}
\big|\widetilde{\,f\,\,}(t,k,\xi)-\widetilde{\,f\,\,}(t',k,\xi)\big|&\lesssim |k|\int_t^{t'}\int_{\R}e^{-0.5\langle\xi-\zeta\rangle^{(1+s)/2}}  \left|\widetilde{\,\,\Theta_k}(\zeta-k\tau,\zeta)\right|\,d\tau d\zeta\\
&\lesssim \int_{t}^{t'}\int_{\R}e^{-0.5\delta_0\langle\xi-\zeta\rangle^{(1+s)/2}} \frac{|k|}{k^2+(\zeta-k\tau)^2}e^{-\lambda\langle k,\zeta\rangle^s}\,d\tau d\zeta,
\end{split}
\end{equation}
from which \eqref{U5} follows. Theorem \ref{thm} is then proved.
\end{proof}

\appendix
\section{Gevrey spaces and Gevrey bounds on the Green's function}\label{appendix}

\subsection{Gevrey spaces}\label{appendix}

We review first some general properties of the Gevrey spaces of functions. 

We start with a characterization of the Gevrey spaces on the physical side. See Lemma A2 in \cite{IOJI} for the elementary proof.

\begin{lemma}\label{lm:Gevrey}
(i) Suppose that $0<s<1$, $K>1$, and $g\in C^{\infty}(\mathbb{T}\times \mathbb{R})$ with ${\rm supp}\,g\subseteq \mathbb{T}\times[-L,L]$ satisfies the bounds
\begin{equation}\label{growth}
\big|D^{\alpha}g(x)\big|\leq K^{m}(m+1)^{m/s},
\end{equation}
for all integers $m\ge 0$ and multi-indeces $\alpha$ with $|\alpha|=m$. Then
\begin{equation}\label{gevreyP}
\big|\widetilde{g}(k,\xi)\big|\lesssim_{K,s} Le^{-\mu|k,\xi|^s},
\end{equation}
for all $k\in\mathbb{Z}, \xi\in \R$ and some $\mu=\mu(K,s)>0$.

(ii) Conversely, assume that $\mu>0$, $s\in(0,1)$, and $g:\T\times\R\to\mathbb{C}$ satisfies
\begin{equation}\label{eq:fouP}
\big\|g\big\|_{\mathcal{G}^{\mu,s}(\mathbb{T}\times \mathbb{R})}\leq 1.
\end{equation}
Then there is $K=K(s,\mu)>1$ such that, for any $m\geq 0$ and all multi-indices $\alpha$ with $|\alpha|\leq m$,
\begin{equation}\label{eq:four}
\left|D^{\alpha}g(x)\right|\lesssim_{\mu,s} K^m(m+1)^{m/s}.
\end{equation}
\end{lemma}

The physical space characterization of Gevrey functions is useful when studying compositions and algebraic operations of functions. For any domain $D\subseteq\T\times\R$ (or $D\subseteq\R$) and parameters $s\in(0,1)$ and $M\geq 1$ we define the spaces
\begin{equation}\label{Gevr2}
\widetilde{\mathcal{G}}^{s}_M(D):=\big\{g:D\to\mathbb{C}:\,\|g\|_{\widetilde{\mathcal{G}}^{s}_M(D)}:=\sup_{x\in D,\,m\geq 0,\,|\alpha|\leq m}|D^\alpha g(x)|M^{-m}(m+1)^{-m/s}<\infty\big\}.
\end{equation}

\begin{lemma}\label{GPF} (i) Assume  $s\in(0,1)$, $M\geq 1$, and $g_1,g_2\in \widetilde{\mathcal{G}}^{s}_M(D)$. Then $g_1g_2\in\widetilde{\mathcal{G}}^{s}_{M'}(D)$ and
\begin{equation*}
\|g_1g_2\|_{\widetilde{\mathcal{G}}^{s}_{M'}(D)}\lesssim \|g_1\|_{\widetilde{\mathcal{G}}^{s}_{M}(D)}\|g_2\|_{\widetilde{\mathcal{G}}^{s}_{M}(D)}
\end{equation*}
for some $M'=M'(s,M)\geq M$. Similarly, if $g_1\geq 1$ in $D$ then $\|(1/g_1)\|_{\widetilde{\mathcal{G}}^{s}_{M'}(D)}\lesssim 1$.

(ii) Suppose $s\in(0,1)$, $M\geq 1$, $I_1\subseteq \R$ is an interval, and $g:\mathbb{T}\times I_1\to \mathbb{T}\times I_2$ satisfies
\begin{equation}\label{gbo1}
|D^\alpha g(x)|\leq M^m(m+1)^{m/s}\qquad \text{ for any }x\in\T\times I_1,\,m\geq 1,\text{ and }|\alpha|\in [1,m].
\end{equation}
If $K\geq 1$ and $h\in \widetilde{G}^s_K(\T\times I_2)$ then $h\circ g\in \widetilde{G}^s_L(\T\times I_1)$ for some $L=L(s,K,M)\geq 1$ and
\begin{equation}\label{Ffgc}
\left\|h\circ g\right\|_{\widetilde{G}^s_L(\T\times I_1)}\lesssim_{s,K,M} \left\|h\right\|_{\widetilde{G}^s_K(\T\times I_2)}.
\end{equation}

(iii) Assume $s\in(0,1)$, $L\in[1,\infty)$, $I,J\subseteq\mathbb{R}$ are open intervals, and $g:I\to J$ is a smooth bijective map satisfying, for any $m\geq 1$,
\begin{equation}\label{gbo2}
|D^\alpha g(x)|\leq L^m(m+1)^{m/s}\qquad \text{ for any }x\in I\text{ and }|\alpha|\in [1,m].
\end{equation}
If $|g'(x)|\geq \rho>0$ for any $x\in I$ then the inverse function $g^{-1}:J\to I$ satisfies the bounds
\begin{equation}\label{gbo2.1}
|D^\alpha (g^{-1})(x)|\leq M^m(m+1)^{m/s}\qquad \text{ for any }x\in J\text{ and }|\alpha|\in [1,m],
\end{equation}
for some constant $M=M(s,L,\rho)\geq L$.
\end{lemma}

Lemma \ref{GPF} 
can be proved by elementary means using just the definition \eqref{Gevr2}. See also Theorem 6.1 and Theorem 3.2 of \cite{Yamanaka} for more general estimates on functions in Gevrey spaces.

\subsubsection{Gevrey cutoff functions}\label{GevCut} Using Lemma \ref{lm:Gevrey}, one can construct explicit cutoff functions in Gevrey spaces. For $a>0$ let
\begin{equation}\label{gev1}
\psi_a(x):=\begin{cases}
e^{-[1/x^a+1/(1-x)^a]}&\quad\text{ if }x\in[0,1],\\
0&\quad\text{ if }x\notin[0,1].
\end{cases}
\end{equation}
Clearly $\psi_a$ are smooth functions on $\R$, supported in the interval $[0,1]$ and independent of the periodic variable. It is easy to verify that $\psi_a$ satisfies the bounds \eqref{growth} for $s:=a/(a+1)$. Thus
\begin{equation}\label{gev2}
|\widetilde{\psi_a}(\xi)|\lesssim e^{-\mu|\xi|^{a/(a+1)}}\qquad\text{ for some }\mu=\mu(a)>0.
\end{equation} 

One can also construct compactly supported Gevrey cutoff functions which are equal to $1$ in a given interval. Indeed, for any $\rho\in[9/10,1)$, the function
\begin{equation}\label{gev3}
\psi'_{a,\rho}(x):=\frac{\psi_a(x)}{\psi_a(x)+\psi_a(x-\rho)+\psi_a(x+\rho)}
\end{equation}
is smooth, non-negative, supported in $[0,1]$, and equal to $1$ in $[1-\rho,\rho]$. Moreover, it follows from Lemma \ref{lm:Gevrey} (i) that $|\widetilde{\psi'_{a,\rho}}(\xi)|\lesssim e^{-\mu|\xi|^{a/(a+1)}}$ for some $\mu=\mu(a,\rho)>0$.

\subsection{Gevrey bounds on the localized Green's function}\label{Gr}
We now provide the proof of Lemma \ref{F18} (which we recall below).
\begin{lemma}\label{F18'}
Define the localized Green's function $\mathcal{G}_k$ as
\begin{equation}\label{F19'}
\mathcal{G}^{\ast}_k(v,w):=\Psi(v)\mathcal{G}_k(v,w)\Psi(w),\qquad {\rm for}\,\,v,w\in\R. 
\end{equation}
Then for some $\delta_0:=\delta_0(\vartheta_1)>0$, we have the bounds
\begin{equation}\label{F20'}
\big|\widetilde{\mathcal{G}_k^{\ast}}(\xi,\eta)\big|\lesssim \frac{e^{-\delta_0\langle\xi+\eta\rangle^{(s+1)/2}}}{k^2+|\eta|^2},\qquad{\rm for}\,\,\xi,\eta\in\R.
\end{equation}
\end{lemma}

\begin{proof}
In view of the definitions \eqref{eq:GreenFunction}, for $y,z\in[0,1]$
\begin{equation}\label{F19.1}
G_k(y,z)=\frac{1}{4|k|\sinh{|k|}}\left[e^{|k|}e^{-|k||y-z|}+e^{-|k|}e^{|k||y-z|}-e^{|k|}e^{-|k|(y+z)}-e^{-|k|}e^{|k|(y+z)}\right].
\end{equation}
Therefore, by direct computation, we conclude that for any $\delta>0$ the following bounds hold:
\begin{equation}\label{F19.2}
\left|\partial_y^m\partial_z^lh_k(y,z)\right|\lesssim \delta^{-m-l}\big[(m+l)!\big]e^{-(\delta/20)|k|},\qquad{\rm for}\,\,k,l\in\mathbb{Z}\cap [0,\infty),
\end{equation}
where either $h_k\in\big\{e^{-|k||y-z|}, e^{-2|k|}e^{|k||y-z|}\big\}$ and $y,z\in[0,1], |y-z|>\delta$; or  $h_k\in\big\{e^{-|k|(y+z)}$, $e^{-2|k|}e^{|k|(y+z)}\big\}$ and $y,z\in[\delta,1-\delta]$; or $h_k=|k|G_k$ and $y,z\in[\delta,1-\delta]$, $|y-z|>\delta$.
 
Notice that
\begin{equation}\label{F21}
b^{-1}(v)-b^{-1}(w)=(v-w)\int_0^1\big(b^{-1}\big)'(w+s(v-w))\,ds.
\end{equation}
Denote
\begin{equation}\label{F22}
F(v,w):=\int_0^1\big(b^{-1}\big)'(w+s(v-w))\,ds,\qquad F^{\ast}(v,w):=b^{-1}(v)+b^{-1}(w).
\end{equation}
Then by \eqref{B}-\eqref{B1}, the physical space characterization of Gevrey spaces, see Lemma \ref{lm:Gevrey} and Lemma \ref{GPF}, $g\in\{F,F^{\ast}\}$ satisfies for $v,w\in[\lv,\uv]$
\begin{equation}\label{F22.1}
\left|\partial_v^m\partial_w^lg(v,w)\right|\lesssim C^{m+l}((m+l)!)^{2/(s+1)},\qquad {\rm for}\,\,m,l\in\mathbb{Z}\cap[0,\infty).
\end{equation}
In view of the change of variables \eqref{F15} and the definitions \eqref{F22}, we can write for $v,w\in[\lv,\uv]$
\begin{equation}\label{F20.0}
\begin{split}
\mathcal{G}_k(v,w)&=\frac{1}{4|k|\sinh{|k|}}\Big[e^{-|k|}e^{|k||b^{-1}(v)-b^{-1}(w)|}+e^{|k|}e^{-|k||b^{-1}(v)-b^{-1}(w)|}\\
                              &\hspace{1in}-e^{|k|}e^{-|k|(b^{-1}(v)+b^{-1}(w))}-e^{-|k|}e^{|k|(b^{-1}(v)+b^{-1}(w))}\Big]\\
                              =\frac{1}{4|k|\sinh{|k|}}&\Big[e^{-|k|}e^{|k|F(v,w)|v-w|}+e^{|k|}e^{-|k|F(v,w)|v-w|}-e^{|k|}e^{-|k|F^{\ast}(v,w)}-e^{-|k|}e^{|k|F^{\ast}(v,w)}\Big].
\end{split}
\end{equation}
Therefore for any $v,w\in\R$,
\begin{equation}\label{F22.2}
\begin{split}
\Psi(v)\mathcal{G}_k(v,v+w)\Psi(v+w)=\frac{\Psi(v)\Psi(v+w)}{4|k|\sinh{|k|}}\Big[&e^{-|k|}e^{|k|F(v,v+w)|w|}+e^{|k|}e^{-|k|F(v,v+w)|w|}\\
&-e^{|k|}e^{-|k|F^{\ast}(v,v+w)}-e^{-|k|}e^{|k|F^{\ast}(v,v+w)}\Big].
\end{split}
\end{equation}
We claim the following bounds for $H_1=e^{-|k||w|F(v,v+w)}, H_2=e^{-2|k|+|k||w|F(v,v+w)}, $ and \\
$H_3\in\big\{-e^{-|k|F^{\ast}(v,v+w)},-e^{-2|k|}e^{|k|F^{\ast}(v,v+w)}\big\}$ with $(v,w)\in{\rm supp}\,\Psi(v)\Psi(v+w)$:
\begin{equation}\label{F23.1}
\left|\partial_v^lH_a(v,w)\right|\lesssim e^{-\mu|kw|}C^l(l!)^{2/(1+s)},\qquad {\rm for}\,\,l\in\mathbb{Z}\cap[0,\infty), a\in\{1,2,3\},
\end{equation}
where $\mu>0$ is a sufficiently small number (depending on $\vartheta_1$). 

The bounds \eqref{F23.1}  for $H_3$ follows from the bounds \eqref{F19.2} for the functions $e^{-|k|(y+z)}$ and $e^{|k|(y+z-2)}$ where $y,z\in[\delta,1-\delta]$ for a sufficiently small $\delta>0$, and the property of Gevrey regular functions under compositions, see Lemma \ref{GPF}. 

To prove the bounds \eqref{F23.1} for $H_1$ we consider separately the cases $|kw|<1$ and $|kw|>1$, and use \eqref{F19.1}-\eqref{F22.1}. More precisely, if $|kw|<1$, the bounds \eqref{F23.1} for $H_1$ follow direclty from  the property of Gevrey regular functions under composition, since $|kw|F(v,v+w)$ is Gevrey regular in $v$ with uniform bounds in $k$ and $w$ satisfying $|kw|<1$. For the case of $|kw|>1$, we first note the function $a\to e^{-\kappa a}$ is uniformly Gevrey regular with respect to $\kappa>1$ in $a\in[\delta,\infty)$ for any fixed $\delta>0$. More precisely, we have for $\kappa>1$,
\begin{equation}\label{F26}
\left|\partial_a^me^{-\kappa a}\right|\lesssim 4^m \delta^{-m}(m!) e^{-\kappa \delta/2},\qquad{\rm for \,\,}m\in\mathbb{Z}\cap [1,\infty).
\end{equation}
Now we view $H_1$ as the composition of $e^{-\kappa a}$ and $F(v,v+w)$ with $\kappa=|kw|$, notice from \eqref{F22} that $F(v,v+w)\approx_{\vartheta_1} 1$ on the support of $\Psi(v)\Psi(v+w)$. Then the bounds \eqref{F23.1} for $H_1$ follow from \eqref{F26} and the property of Gevrey spaces under compositions, see Lemma \ref{GPF}.

The bounds \eqref{F23.1} for $H_2$ follow from similar arguments as the case of $H_1$. The case $|kw|<1$ follows from the same argument so focus on the case $|kw|>1$. We view $H_2$ as the composition of the function $a\to e^{-|k|a}$ and $2-F(v,v+w)|w|$ and notice that $F(v,v+w)|w|<2-\delta$ on the support of $\Psi(v)\Psi(v+w)$ for a small $\delta$ depending on $\vartheta_1$, see \eqref{F21}. The bounds  \eqref{F23.1} for $H_2$ then follow from analogous arguments as in the case of $H_1$. This completes the proof of \eqref{F23.1}.

To finish the proof of Lemma \ref{F18'}, we make the observation that for $\xi,\eta\in\R$,
\begin{equation}\label{F23}
\begin{split}
\int_{\R^2}\Psi(v)\mathcal{G}_k(v,w)\Psi(w)e^{-i v\xi-iw\eta}dv dw=\int_{\R^2}\Psi(v)\mathcal{G}_k(v,v+w)\Psi(v+w)e^{-iv(\xi+\eta)-iw\eta}\,dvdw.
\end{split}
\end{equation}
The claimed bounds \eqref{F20'} follow from integration by parts in \eqref{F23} in the variable $w$ (twice) and then apply Lemma \ref{lm:Gevrey} in the variable $v$, using \eqref{F22.2}-\eqref{F23.1}.
\end{proof}

\end{document}